\documentclass[11pt,reqno,a4paper]{amsart}
\usepackage{graphicx,color}

\usepackage{amsmath,amsfonts,amssymb,amsthm,mathrsfs,amsopn,bm}
\usepackage{latexsym}

\usepackage[numbers]{natbib}

\usepackage{hyperref}

\usepackage[usenames,dvipsnames,x11names,table]{xcolor}

\allowdisplaybreaks

\numberwithin{equation}{section}

\def\H{\mathcal H}

\def\R{\mathbb R}
\def\N{\mathbb N}
\def\Z{\mathbb Z}

\def\e{\varepsilon}

\def\00{{\bf 0}}

\def\P{\mathcal{P}}

\newcommand{\vett}[1]{\boldsymbol{#1}}

\newcommand\res{\mathop{\hbox{\vrule height 7pt width .3pt depth 0pt \vrule height .3pt width 5pt depth 0pt}}\nolimits}

\def\FF{\mathbf{F}}

\theoremstyle{plain}
\newtheorem{theorem}{Theorem} [section]
\newtheorem{corollary}[theorem]{Corollary}
\newtheorem{lemma}[theorem]{Lemma}
\newtheorem{proposition}[theorem]{Proposition}
\theoremstyle{definition}

\title[On the anisotropic Kirchhoff-Plateau problem]{On the anisotropic Kirchhoff-Plateau problem}

\author{Antonio De Rosa}
\address[A.\,De Rosa]{Courant Institute of Mathematical Sciences, New York University, 251 Mercer Street, New York 10012, NY, USA}
\email{derosa@cims.nyu.edu}

\author{Luca Lussardi}
\address[L.\,Lussardi]{Dipartimento di Scienze Matematiche ``G.L. Lagrange'', Politecnico di Torino, C.so Duca degli Abruzzi, 10129 Torino, Italy}
\email{luca.lussardi@polito.it}

\begin{document}

\baselineskip3.4ex

\vspace{0.5cm}
\begin{abstract}
{\small We extend to the anisotropic setting the existence of solutions for the Kirchhoff-Plateau problem and its dimensional reduction.
}
\end{abstract}

\maketitle

\section{Introduction}

In this article we focus on the minimizers of the anisotropic Kirchhoff-Plateau problem. These correspond to soap films that span flexible rods under the action of general energies which are not translation or rotation invariant. The corresponding isotropic problem of minimizing the area functional has been investigated by Giusteri {\em et al.} in \cite{GLF} with only one filament, and by Bevilacqua {\em et al.} in \cite{BLM2} taking into account more complex configurations. An increasing interest has been devoted to the study of the anisotropic Plateau problem, see for instance \cite{DePDeRGhi,DePDeRGhiCPAM,DePDeRGhi2,DeRosa,DRK,HarrisonPugh16}. To model the flexible rods, we impose physical constraints, as for instance local and global non-interpenetration of matter, introduced by Schuricht in \cite{S}. Moreover we add the necessary specifications in considering a link rather than a single loop. The energy functional we minimize is given by the sum of the elastic and the potential energy for the link and the anisotropic surface energy of the film. Concerning the boundary condition, we use the definition introduced by Harrison in \cite{HarrisonPugh14}, based on the concept of linking number, which is a well-known topological invariant. For the poof in Section \ref{s2}, we rely on the result by De Lellis {\em et al.}\ \cite{DDG}, who formulate the anisotropic Plateau problem in fairly general spanning conditions. To conclude, in Section \ref{s3} we perform a dimensional reduction of the aforementioned variational problem, in the spirit of the analysis carried out in the isotropic setting in \cite{BLM3}.

\section{Notation and preliminaries}

In this section we recall notation for the geometry of curves. If $\vett x_1,\vett x_2 \colon [0,L] \to \R^3$ are two continuous and closed curves, their {\it linking number} is the integer value 
\[
{\rm Link}(\vett x_1,\vett x_2):=\frac{1}{4\pi}\int_0^L\int_0^L\frac{{\vett x_1}(s)-{\vett x_2}(t)}{|\vett x_1(s)-\vett x_2(t)|^3}\cdot \vett x_1'(s)\times \vett x_2'(t)\,dsdt.
\]
We say that $\vett x_1$ and $\vett x_2$ are {\it isotopic}, and we use the notation $\vett x_1\simeq\vett x_2$, if there exists an open neighborhood $N_1$ of $\vett x_1([0,L])$, an open neighborhood $N_2$ of $\vett x_2([0,L])$ and a continuous map $\Phi\colon N_1\times[0,1]\to\R^3$ such that $\Phi(N_1,\tau)$ is homeomorphic to $N_1$ for all $\tau$ in $[0,1]$ and
\[
\Phi(\cdot,0)=\mathrm{Identity}\,,\quad\Phi(N_1,1)=N_2\,,\quad\Phi(\vett x_1([0,L]),1)=\vett x_2([0,L])\,.
\]
Following Gonzalez et al.\,\cite{GMSM}, we define the {\it minimal global radius of curvature} of a closed curve $\vett x\in W^{1,p}([0,L];\R^3)$, with $p>1$, by
\[
\Delta(\vett x):=\inf_{s\in [0,L]}\inf_{\sigma,\tau \in [0,L]\setminus \{s\}}R(\vett x(s),\vett x(\sigma),\vett x(\tau))
\]
where $R(x,y,z)$ denotes the radius of the smallest circle containing $x,y,z$, with the convention $R(x,y,z)=+\infty$ if $x,y,z$ are collinear. The global radius of curvature determines the self-intersections of the tubular neighborhoods of a curve. More precisely, for every $r>0$ we define the {\it $r$-tubular neighborhood of $\vett x$} by
\[
U_r(\vett x)=\bigcup_{s\in [0,L]}B_r(\vett x(s)).
\]
Accordingly to Ciarlet et al.\,\cite{CN} we say that $U_r(\vett x)$ is {\it not self-intersecting} if for any $p\in \partial U_r(\vett x)$ there exists a unique $s\in [0,L]$ such that $\|p-\vett x(s)\|=r$. It turns out  (see Gonzalez et al.\,\cite{GMSM}) that $\Delta(\vett x)\ge r$ if and only if $U_r(\vett x)$ is not self-intersecting. In particular, if $\Delta(\vett x)>0$ then $\vett x$ is simple, that is $\vett x\colon [0,L)\to \R^3$ is injective. 

\section{The anisotropic Plateau problem}

First we recall that a set \(S \subset \R^3\) is said to be {\it \(2\)-rectifiable} if it can be covered, up to an \(\H^2\)-negligible set, by countably many $2$-dimensional submanifolds of class $C^1$, see \cite[Chapter 3]{SimonLN}; we also denote by $G$ the Grassmannian of unoriented $2$-dimensional planes in $\R^3$. Given a $2$-rectifiable set $S$, we denote by $T_xS$ the approximate tangent space of $S \subset \R^3$ at $x$, which exists for $\H^2$-almost every point $x \in S$ \cite[Chapter 3]{SimonLN}. The anisotropic Lagrangians considered in the rest of the note will be continuous maps 
$$
F\colon \R^3\times G\ni (x,\pi)\mapsto F(x,\pi) \in (0,+\infty) ,
$$
verifying the lower and upper bounds 
\begin{equation}\label{cost per area}
0 < \lambda \leq F(x,\pi) \leq \Lambda<\infty \qquad \forall (x,\pi)\in \R^3\times G.
\end{equation}
We also require that $F$ is {\it elliptic} {\cite[5.1.2-5.1.5]{FedererBOOK}}, that is its even and positively $1$-homogeneous extension to $\mathbb R^3\times (\Lambda_2 (\mathbb R^3)\setminus \{0\})$ is $C^2$ and it is convex in the $\pi$ variable. Given a $2$-rectifiable set $S$ and an open subset $U\subset \R^3$, we define:
\begin{equation}\label{energia}
\FF(S) := \int_S F(x,T_xS)\, d\H^2(x).
\end{equation}
Next, we need to define the spanning condition. For any $H\subset \R^3$ closed let $\mathcal C(H)$ be the class of all smooth embeddings $\gamma \colon {\mathbb{S}}^1 \to \R^3 \setminus H$. Given $\mathcal C \subset \mathcal C(H)$ closed by homotopy, namely if $\gamma \in \mathcal C$ then also $\tilde\gamma\in \mathcal C$ for any $\tilde\gamma\in [\gamma]\in \pi_1(\R^3 \setminus H)$, we denote by $\P(H,\mathcal C)$ the family of all 2-rectifiable relatively closed sets $S \subset  \R^3 \setminus H$ such that
\[
S \cap \gamma(\mathbb S^1)\ne \emptyset, \quad \forall \gamma \in \mathcal C.
\]
We recall the following result, see \cite[Theorem 2.7]{DDG}:
\begin{theorem}\label{thm plateau}
The problem
\[
\min\{\FF(S) : S \in \P(H,\mathcal C)\}
\]
has a solution $S\in\P (H,\mathcal C)$ and the set $S$ is an $(\FF , 0, \infty)$-minimal set in $\R^3\setminus H$ in the sense of Almgren {\rm\cite{Almgren76}}.
\end{theorem}

\section{The anisotropic Kirchhoff-Plateau problem}\label{s2}

\subsection{The system of linked rods}\label{par-rod}

Let $N\in \N$ and $p\in (1,+\infty)$. For every $i=1,\dots,N$, let $L^i>0$ and $\vett x_0^i,\vett t^i_0,\vett d^i_0 \in \R^3$ be such $\vett t^i_0\perp \vett d^i_0$ and $|\vett t^i_0|=|\vett d^i_0|=1$. Moreover let $\kappa^i_1,\kappa^i_2,\omega^i\in L^p([0,L^i])$ such that 
\[
w^i_1:=(\kappa_1^i,\kappa_2^i,\omega^i) \in L^p([0,L^i];\R^3), \quad w^i:=(w^i_1,\vett x_0^i,\vett t_0^i,\vett d_0^i) \in L^p([0,L^i];\R^3) \times \R^3\times \R^3\times \R^3,
\]
and 
\[
w:=(w^1_1,w^2,\dots,w^N) \in  L^p([0,L^1];\R^3) \times \prod_{i=2}^N((L^p([0,L^i];\R^3) \times \R^3 \times \R^3 \times \R^3)=:V.
\]
We endow $V$ with the natural $p$-norm, that we denote by $\|\cdot\|_V$. For any $i=1,\dots,N$ and for any $w\in V$ denote by $\vett x^i[w] \in W^{2,p}([0,L^i];\R^3)$ and $\vett t^i[w],\vett d^i[w]\in W^{1,p}([0,L^i];\R^3)$ the unique solutions of the Cauchy problem
\[
\left\{\begin{aligned}
&\vett x^i[w]'(s)=\vett t^i[w](s)\\
&\vett t^i[w]'(s)=\kappa^i_1(s)\vett d^i[w](s)+\kappa^i_2(s)\vett t^i[w](s)\times\vett d^i[w](s)\\
&\vett d^i[w]'(s)=\omega^i(s)\vett t^i[w](s)\times\vett d^i[w](s)-\kappa^i_1(s)\vett t^i[w](s)\\
&\vett x^i[w](0)=\vett x^i_0\\
&\vett t^i[w](0)=\vett t^i_0\\
&\vett d^i[w](0)=\vett d^i_0.
\end{aligned}\right.
\]
It is easy to see that $\vett t^i[w](s)\perp \vett d^i[w](s)$ and $|\vett t^i[w](s)|=|\vett d^i[w](s)|=1$ and consequently that 
\[
(\vett t^i[w](s),\vett d^i[w](s),\vett t^i[w](s)\times \vett d^i[w](s))
\]
is an orthonormal frame in $\R^3$, for any $s\in [0,L^i]$ and for any $i=1,\dots,N$. Let $\eta,\nu>0$ and consider  $\mathcal A^i(s) \subset \R^2$ be compact and simply connected such that 
\[
B_\eta(\vett 0)\subset \mathcal A^i(s)\subset B_\nu(\vett 0), \qquad \forall s\in [0,L^i],\,  i=1,\dots,N.
\] 
For any $i=1,\dots,N$ we define
\[
\Omega^i:=\{(s,\zeta_1,\zeta_2) \in \R^3: s\in[0,L^i],\,(\zeta_1,\zeta_2)\in\mathcal A^i(s) \},
\]
\begin{equation}\label{rod}
\Lambda^i[w]:=\{\vett x^i[w](s)+\zeta_1\vett d^i[w](s)+\zeta_2\vett t^i[w](s)\times\vett d^i[w^i](s) : (s,\zeta_1,\zeta_2)\in \Omega^i\},
\end{equation}
and 
\[
\Lambda[w]:=\bigcup_{i=1}^N\Lambda^i[w].
\]
The system of closed rods is subject to some constraints. First of all we assume that the midlines are closed and sufficiently smooth, that is 
\begin{itemize}
\item[\rm(C1)] $\vett x^i[w](L^i)=\vett x^i[w](0)=\vett x^i_0$, for any $i=1,\dots,N$
\end{itemize}
and
\begin{itemize}
\item[\rm(C2)] $\vett t^i[w](L^i)=\vett t^i[w](0)=\vett t^i_0$, for any $i=1,\dots,N$.
\end{itemize}
To prescribe how many times the ends of the rod are twisted before being glued together, we prescribe the linking number between the midline and a closed curve close to the midline. Precisely, for any $i=1,\dots,N$ we close up the curve $\vett x^i[w]+\tau\vett d^i[w]$, for $\tau>0$ fixed and small enough, defining, as in Schuricht \cite{S}, 
\begin{equation}\label{xtau}
\begin{aligned}
{\vett \tilde x^i}_\tau&[w](s)\\
&:=\left\{\begin{array}{ll}
\vett x^i[w](s)+\tau\vett d^i[w](s)\\
\qquad \text{if $s\in [0,L^i]$}\\
\vett x^i[w](L^i)+\tau(\cos(\varphi^i(s-L^i))\vett d^i[w](L^i)+\sin(\varphi^i(s-L^i))\vett t^i[w](L^i)\times \vett d^i[w](L^i))\\
\qquad \text{if $s\in [L^i,L^i+1]$}\\
\end{array}\right.
\end{aligned}
\end{equation}
where $\varphi^i\in [0,2\pi)$ is the unique angle between $\vett d^i_0$ and $\vett d^i[w](L^i)$ such that $\varphi^i-\pi$ has the same sign as $\vett d^i_0 \times \vett d^i[w](L^i) \cdot \vett t^i_0$. We trivially identify $\vett x^i[w]$ with its extension $\vett x^i[w](s)=\vett x^i(L^i)$ for any $s\in [L^i,L^i+1]$ and therefore we require that for any $i=1,\dots,N$ there is some $l^i\in \Z$ such that 
\begin{itemize}
\item[\rm(C3)] ${\rm Link}(\vett x^i[w],{\vett \tilde x^i}_\tau[w])=l^i$.
\end{itemize}
To encode the knot type of the midlines for any $i=1,\dots,N$ we fix continuous mapping $\vett\ell^i\colon [0,L^i]\to\R^3$ such that $\vett\ell^i(L^i)=\vett\ell^i(0)$ and we require that 
\begin{itemize}
\item[\rm(C4)] $\vett x^i[w]\simeq\vett\ell^i$.
\end{itemize}
Finally, in order to prevent the interpenetration of matter, following Ciarlet et al.\,\cite{CN} we require that for any $i=1,\dots,N$ 
\begin{itemize}
\item[\rm(C5)]
\[
\int_{\Omega^i}(1- \zeta_1k^i_2(s) + \zeta_2k^i_1(s))\,dsd\zeta_1d\zeta_2 \leq |\Lambda^i[w]|, \quad \bigcap_{i=1}^N{\rm int}(\Lambda^i[w])=\emptyset.
\]
\end{itemize}
We now require that our system of rods has a prescribed chain structure. We assume that:
\begin{itemize}
\item[\rm(C6)] for any $i=1,\dots,N$ there exist $k$ and $i_j$ for $j=1,\dots,k$ such that 
\[
|{\rm Link}(\vett x^{i}[w],\vett x^{i_1}[w])|=|{\rm Link}(\vett x^{i_j}[w],\vett x^{i_{j+1}}[w])|=\cdots=|{\rm Link}(\vett x^{i_k}[w],\vett x^{1}[w])|=1
\] 
and moreover ${\rm Link}(\vett x^{i}[w],\vett x^{j}[w])=L^{ij}$ where $L^{ij}$ do not depend on $w$.
\end{itemize}
We finally denote by $W$ the set of all constraints, namely 
\[
W:=\big\{w\in V: \text{(C1)--(C6) hold true} \big\}.
\]
It is easy to see (see Gonzalez et al.\,\cite{GMSM} and Schuricht \cite{S}) that $W$ is weakly closed in $V$.

\subsection{Energy contributions and existence of a minimizer}

In what follows we will prescribe an elastic energy of the system of rods, which is a proper function 
\begin{equation}\label{elas}
E^\mathrm{el}\colon W \to\R\cup\{+\infty\}, \qquad \text{satisfying} \qquad E_\mathrm{el}(w) \ge c\|w\|_V,
\end{equation}
for some $c>0$. The second energy contribution we want to take into account is the weight of the rods. Let $\rho^i \in L^\infty(\Omega^i)$ with $\rho\ge 0$ be the mass density functions and $\vett g$ be the gravitational acceleration. Let us define $E_\mathrm{g}\colon W \to\R\cup\{+\infty\}$ as 
\[
\begin{aligned}
E_\mathrm{g}&(w):=\sum_{i=0}^N\int_{\Omega^i}\rho^i(s,\zeta_1,\zeta_2) \vett g \cdot (\vett x^i[w](s)+\zeta_1\vett d^i[w](s)+\zeta_2\vett t^i[w](s)\times\vett d^i[w](s))\,dsd\zeta_1d\zeta_2.
\end{aligned}
\]
The last contribution we want to take into account is the soap film energy. Let $\mathcal C_w \subset \mathcal C(\Lambda[w])$ be the class of all $\gamma \in \mathcal C(\Lambda[w])$ such that there exists $i=1,\dots,N$ with 
\[
|{\rm Link}(\gamma,\vett x^i[w])|=1, \quad {\rm Link}(\gamma,\vett x^j[w])=0,\,\forall j\ne i.
\]
$\mathcal C_w$ is closed by homotopy, see \cite{HarrisonPugh14}. We finally define $E_\mathrm{sf}\colon W \to\R\cup\{+\infty\}$ as
\[
E_\mathrm{sf}(w):=\inf\big\{\FF(S) : \textrm{$S \in \mathcal P(\Lambda[w],\mathcal C_w)$}\big\}.
\]
We define the energy functional of our variational problem as
\begin{equation}
\label{EKP}
E(w):= E_{\rm el}(w) + E_{\rm g}(w)+E_{\rm sf}(w), \quad w\in W.
\end{equation}
The first main result of the paper is given by the following existence theorem.

\begin{theorem}\label{main1}
Let $\overline{E_{\rm el}}$ be the lower semicontinuous envelope of $E_{\rm el}$ with respect to the weak topology of $V$. Assume that $\inf_WE<+\infty$. Then the problem 
\[
\min_{w\in W}\overline{E_{\rm el}}(w) + E_{\rm g}(w)+E_{\rm sf}(w)
\]
has a solution $w_0 \in W$ and there exists $S_\infty \in \mathcal P(\Lambda[w_0],\mathcal C_{w_0})$ which is an $(\FF,0,\infty)$-minimal set in $\R^3 \setminus \Lambda[w_0]$ in the sense of Almgren such that 
\[
E_{\rm el}(w_0) + E_{\rm g}(w_0)+\FF(S_\infty)=\min_{w\in W}\overline{E_{\rm el}}(w) + E_{\rm g}(w)+E_{\rm sf}(w)=\inf_{w\in W}E(w).
\]
\end{theorem}

\subsection{Proof of Theorem \ref{main1}}

First of all we prove that the weight and the soap film energy are weakly continuous. 

\begin{lemma}\label{cont1}
The functional $E_{\rm g}$ is weakly continuous on $W$.
\end{lemma}

\begin{proof}
Let $(w_h)$ be a sequence in $W$ with $w_h \rightharpoonup w$ in $W$ for some $w\in W$. Then $\vett x^i[w_h] \rightharpoonup \vett x^i[w]$ in $W^{2,p}$ and $\vett t^i[w_h] \rightharpoonup \vett t^i[w]$, $\vett d^i[w_h] \rightharpoonup \vett d^i[w]$ in $W^{1,p}$. Then by Sobolev embedding we deduce that $\vett x^i[w_h] \to \vett x^i[w]$ in $C^{1,\alpha}$ and $\vett t^i[w_h] \to \vett t^i[w]$, $\vett d^i[w_h] \to \vett d^i[w]$ in $C^{0,\alpha}$. This is enough to pass to the limit under the sign of integral and get the claim.
\end{proof}

The continuity of the  soap film energy follows from the next theorem.

\begin{theorem}
\label{primo2}
Let $(w_h)$ be a sequence in $W$ with $w_h \rightharpoonup w$ in $W$ for some $w\in W$. Assume that
\begin{itemize}
\item[\rm(a)] $S_h\in \mathcal{P}(\Lambda[w_h],\mathcal C_{w_h})$, for every $h\in\N$;
\item[\rm(b)] $\sup_{h \in \N}\FF(S_h) = \sup_{h \in \N} \inf \{\FF(S) : S\in \mathcal{P}(\Lambda[w_h],\mathcal C_{w_h})\}<+\infty.$
\end{itemize}
Let $\mu_h:= F\mathcal H^2\res S_h$. Then the following three statement hold true:
\begin{equation}\label{c11}
\mu_h \rightharpoonup^{*} \mu \quad \textrm{(up to subsequences)},
\end{equation}
\begin{equation}\label{c21}
\mu \geq F \mathcal H^2\res S_{\infty},\hbox{ where $S_{\infty}= (\hbox{supt}\, \mu) \setminus \Lambda[w]$ is $2$-rectifiable};
\end{equation}
\begin{equation}\label{c31}
S_\infty \in \mathcal P(\Lambda[w],\mathcal C_w).
\end{equation}
\end{theorem}

\begin{proof}
We first observe that the classes $\mathcal P(\Lambda[w_h],\mathcal C_{w_h})$ and $\mathcal P(\Lambda[w],\mathcal C_w)$ are good classes in the sense of De Lellis et al.\,\cite[Def.\,2.2]{DDG}, as proved in \cite[Thm.\,2.7(a)]{DDG}. Then the proof of \eqref{c11} and \eqref{c21} follows verbatim the proof of Theorem 2.5 of \cite{DDG}. It is sufficient to observe that the convergence of $\{\Lambda[w_h]\}$ ensures that, whenever $x \in S_{\infty}$, we have $d(x,\Lambda[w_h])>0$ for $h$ large enough. We have to prove \eqref{c31}, namely that $S_{\infty}\cap \gamma(\mathbb S^1)\ne \emptyset$ for any $\gamma \in \mathcal C_w$. Assume by contradiction that there exists $\gamma \in \mathcal C_w$ with $S_{\infty}\cap\gamma(\mathbb S^1)= \emptyset$. Since $\gamma$ is compact and contained in $\R^3 \setminus \Lambda[w]$ and $S_\infty$ is relatively closed in $\R^3 \setminus \Lambda[w]$, there exists a positive $\e$ such that the tubular neighborhood $U_{2\varepsilon}(\gamma)$ does not intersect $S_{\infty}$ and is contained in $\R^3 \setminus \Lambda[w]$. Hence $\mu(U_{2\varepsilon}(\gamma)) = 0$, and thus
\begin{equation}\label{utile}
\lim_h\mathcal H^2(S_h \cap U_{\varepsilon}(\gamma))=0.
\end{equation}
Denote by $B_\varepsilon$ the open disk of $\R^2$ with radius $\varepsilon$ and centered at the origin of $\R^2$, and consider a diffeomorphism $\Phi\colon \mathbb S^1 \times B_\varepsilon \to U_{\varepsilon}(\gamma)$ such that $\Phi_{|_{\mathbb S^1\times \{0\}}}=\gamma$. Let $y$ belong to $B_\varepsilon$ and set $\gamma_{y}:=\Phi_{|_{\mathbb S^1\times \{y\}}}$. Then $\gamma_{y}$ in $[\gamma]$ represents an element of $\pi_1(\mathbb R^3 \setminus \Lambda[w])$.  Since $w_h\rightharpoonup w$ in $W$ then $\{\vett x^i[w_h]\}$ converges to $\vett x^i[w]$ strongly in $W^{1,p}([0,L];\mathbb R^3)$ for every $i=1,\dots,N$. In particular, $\{\vett x^i[w_h]\}$ converges to $\vett x^i[w]$ uniformly on $[0,L^i]$  for every $i=1,\dots,N$, which implies the existence of $\delta>0$ such that, for $h$ sufficiently large, $\Lambda[w_h]$ is contained in $U_\delta(\Lambda[w])$ with $U_\delta(\Lambda[w])\cap U_\varepsilon(\gamma)=\emptyset$. Hence, for such $h$ and $\varepsilon$ it follows that, for any $y$ in $B_\varepsilon$, $\gamma_{y}(\mathbb S^1) \subset \mathbb R^3\setminus U_\delta(\Lambda[w])$. This implies that $\|\vett x^i[w_h]-\gamma_y\|_\infty\geq \delta$  for every $i=1,\dots,N$. This estimate, together with the $W^{1,p}$ convergence of $\vett x^i[w_h]$ to $\vett x^i[w]$, implies that
\[
\lim_{h \to +\infty}\text{Link}(\vett x^i[w_h], \gamma_y)= \text{Link}(\vett x^i[w], \gamma_y), \qquad \forall i=1,\dots,N. 
\]
As a consequence, for $h$ large enough, $\gamma_y \in \mathcal C_{w_h}$ which, combined with $S_h\in \mathcal P(\Lambda[w_h],\mathcal C_{w_h})$, yields $S_h\cap \gamma_{y}(\mathbb S^1)\ne \emptyset$. Take now $\tilde \pi : \mathbb S^1 \times B_\varepsilon \to B_\varepsilon$ as the projection on the second factor and let $\hat{\pi}:=\tilde \pi \circ \Phi^{-1}$. Then, $\hat{\pi}$ is Lipschitz-continuous and $B_\varepsilon$ is contained in $\hat{\pi}(S_h \cap U_\varepsilon(\gamma))$, which entails that
\[
\pi \varepsilon^2=\mathcal H^2(B_\varepsilon)\le\mathcal H^2(\hat{\pi}(S_h \cap U_\varepsilon(\gamma))\le ({\rm Lip}\,\hat{\pi})^2\mathcal H^2(S_h \cap U_\varepsilon(\gamma))\,.
\]
We thus conclude that
\[
\mathcal H^2(S_h \cap U_\varepsilon(\gamma))\ge \frac{\pi \varepsilon^2}{({\rm Lip}\,\hat{\pi})^2}
\]
 which contradicts \eqref{utile}.
\end{proof}

{\it Proof of Theorem \ref{main1}}
First of all thanks to the weak continuity of $E_{\rm g}$ and $E_{\rm sf}$, proved in Lemma \ref{cont1} and Theorem \ref{primo2}, we deduce that $\overline{E_{\rm el}}(w) + E_{\rm g}(w)+E_{\rm sf}(w)$ is the lower semicontinuous envelope of $E$, from which we get
\[
\inf_{w\in W}\overline{E_{\rm el}}(w) + E_{\rm g}(w)+E_{\rm sf}(w)=\inf_{w\in W}E(w).
\]
Let $\{w_h\}$ be a minimizing sequence for $E_{\rm el}+ E_{\rm g}+E_{\rm sf}$. Since $\inf_WE<+\infty$ we can say that $E(w_h)\le c$ for some $c>0$. In particular, $E_{\rm el}(w_h)\le c$ and, by coercivity of $E_{\rm el}$, we have $w_h\rightharpoonup w_0$ in $W$. We deduce, using again Lemma \ref{cont1} and Theorem \ref{primo2}, that
\[
\begin{aligned}
\overline{E_{\rm el}}(w_0) + E_{\rm g}(w_0)+E_{\rm sf}(w_0)&\le \liminf_h\overline{E_{\rm el}}(w_h) + E_{\rm g}(w_h)+E_{\rm sf}(w_h)\\
&\le \liminf_h E(w_h)=\inf_WE=\inf_W\overline{E_{\rm el}} + E_{\rm g}+E_{\rm sf}.
\end{aligned}
\]
Moreover, since $E_{\rm sf}(w_0)<+\infty$, applying Theorem 2.7 of \cite{DDG} we deduce the claim.\qed

\section{Dimensional reduction of the anisotropic Kirchhoff-Plateau problem}\label{s3}

The second main result of the paper concerns the dimensional reduction. In this section we focus on a simplified setting with a single rod which has a cross section with vanishing diameter; moreover we also need to modify the constraints. For the sake of convenience we briefly rewrite the complete setting. Let $L>0$, $p\in (1,+\infty)$ and let $\kappa_1,\kappa_2,\omega\in L^p([0,L])$. Let $\vett x_0,\vett t_0,\vett d_0 \in \R^3$ be such $\vett t_0\perp \vett d_0$ and $|\vett t_0|=|\vett d_0|=1$ and let 
\[
w:=(\kappa_1,\kappa_2,\omega) \in L^p([0,L];\R^3).
\]
Denote by $\vett x[w] \in W^{2,p}([0,L];\R^3)$ and $\vett t[w],\vett d[w]\in W^{1,p}([0,L];\R^3)$ the unique solutions of the Cauchy problem
\[
\left\{\begin{aligned}
&\vett x[w]'(s)=\vett t[w](s)\\
&\vett t[w]'(s)=\kappa_1(s)\vett d[w](s)+\kappa_2(s)\vett t[w](s)\times\vett d[w](s)\\
&\vett d[w]'(s)=\omega(s)\vett t[w](s)\times\vett d[w](s)-\kappa_1(s)\vett t[w](s)\\
&\vett x[w](0)=\vett x_0\\
&\vett t[w](0)=\vett t_0\\
&\vett d[w](0)=\vett d_0.
\end{aligned}\right.
\]
For any $s\in [0,L]$ let $\mathcal A(s) \subset \R^2$ be compact and simply connected such that 
\[
B_\eta(\vett 0)\subset \mathcal A(s)\subset B_\nu(\vett 0), \qquad \forall s\in [0,L],
\] 
for some $\eta,\nu>0$. For any $\e>0$ small enough and for any $w \in L^p([0,L];\R^3)$ let
\begin{equation}\label{rod1}
\Lambda_\e[w]:=\{\vett x[w](s)+\zeta_1\vett d[w](s)+\zeta_2\vett t[w](s)\times\vett d[w](s) : (s,\zeta_1,\zeta_2)\in \Omega_\e\}
\end{equation}
where 
\[
\Omega_\e:=\{(s,\zeta_1,\zeta_2)\in \R^3 : s\in[0,L],\,(\zeta_1,\zeta_2)\in \e \mathcal A(s) \}.
 \]
The constraints are the following.  
\begin{itemize}
\item[\rm(C1)] $\vett x[w](L)=\vett x[w](0)=\vett x_0$.
\end{itemize}
\begin{itemize}
\item[\rm(C2)] $\vett t[w](L)=\vett t[w](0)=\vett t_0$.
\end{itemize}
\begin{itemize}
\item[\rm(C3)] ${\rm Link}(\vett x[w],{\vett \tilde x}_\tau[w])=l$ for some fixed $l\in \Z$, where ${\vett \tilde x}_\tau[w]$ is defined as in \eqref{xtau} (of course without the index $i$). 
\end{itemize}
\begin{itemize}
\item[\rm(C4)] $\vett x[w]\simeq\vett\ell$ for some continuous mapping $\vett\ell\colon [0,L]\to\R^3$ such that $\vett\ell(L)=\vett\ell(0)$.
\end{itemize}
Finally, in order to prevent the non-selfintersection we require that 
\begin{itemize}
\item[\rm(C5)] $\Delta (\vett x[w])\ge \Delta_0$ for some prescribed $\Delta_0>0$.
\end{itemize}
Again, we denote by $W$ the set of all constraints, namely 
\[
W:=\big\{w\in L^p([0,L];\R^3): \text{(C1)--(C5) hold true} \big\}.
\]
It turns out that $W$ is weakly closed in $L^p([0,L];\R^3)$. The main goal is to prove that sending $\e$ to 0, we recover by $\Gamma$-convergence the anisotropic Plateau problem with an elastic one dimensional boundary.  The first two energy contributions to take into account are the elastic energy $E_{\rm el}$ as in \eqref{elas} and the scaled weight 
\[
E^{\rm g}_\e(w):=\frac{1}{\e^2}\int_{\Omega_\e}\rho(s,\zeta_1,\zeta_2) \vett g \cdot (\vett x[w](s)+\zeta_1\vett d[w](s)+\zeta_2\vett t[w](s)\times\vett d[w](s))\,dsd\zeta_1d\zeta_2
\]
where $\rho\in L^\infty(\Omega_1)$ and $\rho \ge 0$. Concerning the soap film energy, similarly to the previous section, we define $\mathcal C_{\e,w} \subset \mathcal C(\Lambda_\e[w])$ as the class of all $\gamma \in \mathcal C(\Lambda_\e[w])$ such that $|{\rm Link}(\gamma,\vett x[w])|=1$ and we define $E^\mathrm{sf}_\e\colon W \to\R\cup\{+\infty\}$ as
\[
E^{\mathrm{sf}}_\e(w):=\inf\big\{\FF(S) : \textrm{$S \in \mathcal P(\Lambda_\e[w],\mathcal C_{\e,w})$}\big\}.
\]
Let $\rho_0\colon [0,L] \to \R$ be given by
\[
\rho_0(s):=\lim_{(\xi_1,\xi_2)\to (0,0)}\rho(s,\xi_1,\xi_2)
\] 
and let 
\[
E_0(w):=\overline{E_{\rm el}}(w)+\int_0^L|\mathcal A(s)|\rho_0(s)\vett g\cdot \vett x[w](s)\,ds+\inf\big\{\FF(S) : \textrm{$S\in \mathcal P(\vett x[w]([0,L]),\mathcal C_w)$}\big\},
\]
where $\mathcal C_w \subset \mathcal C(\vett x[w]([0,L]))$ is the class of all $\gamma \in \mathcal C(\vett x[w]([0,L]))$ such that $|{\rm Link}(\gamma,\vett x[w])|=1$.
We can then define $E_\e \colon W \to \R \cup \{+\infty\}$ for any $\e>0$ let as 
\[
E_\e(w):=E^\mathrm{el}(w)+E^\mathrm{g}_\e(w)+E^\mathrm{sf}_\e(w).
\]
We are ready to state our second main result. 

\begin{theorem}\label{main2}

Let $(\e_h)$ be a sequence such that $\e_h \to 0$ as $h\to +\infty$ and let $(w_h)$ be a sequence in $W$ with $\sup_{h \in \N}E_{\e_h}(w_h)\le c$ for some $c>0$. Then, up to a subsequence, $w_h \rightharpoonup w$ in $L^p([0,L];\R^3)$ and $w\in W$. Moreover, the family $\{E_\e\}_{\e>0}$ $\Gamma$-converges to $E_0$ as $\e\to 0^+$ with respect to the weak topology of $L^p([0,L];\R^3)$, namely:
\begin{itemize}
\item[\rm(a)] for any sequence $(\e_h)$ with $\e_h \to 0$, for any $w\in W$ and for any sequence $(w_h)$ in $W$ with $w_h \rightharpoonup w$ in $L^p([0,L];\R^3)$ we have 
\begin{equation}\label{lower}
E_0(w) \le \liminf_{h\to+\infty} E_{\e_h}(w_h);
\end{equation}
\item[\rm(b)] for any $w\in W$ there is a sequence $(\e_h)$ with $\e_h \to 0$ and a sequence $(\bar w_h)$ in $W$ with $\bar w_h \rightharpoonup w$ in $L^p([0,L];\R^3)$ such that 
\begin{equation}\label{upper}
E_0(w) \ge \limsup_{h\to+\infty} E_{\e_h}(\bar w_h).
\end{equation}
\end{itemize}
\end{theorem}

As a standard consequence of Theorem \ref{main2} we have the next result.

\begin{corollary}\label{cor}
Let $(\e_h)$ be such that $\e_h\to 0$ as $h\to +\infty$. For any $h\in \N$ and for any $\sigma_h\to 0$ let $w_h \in W$ be such that 
\begin{equation}\label{quasimin}
E_{\e_h}(w_h)\le \inf_WE_{\e_h}+\sigma_h.
\end{equation}
Then up to a subsequence $w_h \rightharpoonup w_0$ in $L^p([0,L];\R^3)$ and 
\[
E_0(w_0)=\min_WE_0.
\]
\end{corollary}

\subsection{Proof of Theorem \ref{main2}}

Fix a sequence $\e_h \to 0$ as $h\to +\infty$.

\begin{proposition}\label{prop-comp}
Let $(w_h)$ be a sequence in $W$ with $\sup_{h \in \N}E_{\e_h}(w_h)\le c$ for some $c>0$. Then, up to a subsequence, $w_h \rightharpoonup w$ in $L^p([0,L];\R^3)$ and $w\in W$.
\end{proposition}

\begin{proof}
The conclusion follows from the coercivity of $E_{\rm el}$.
\end{proof}

The study of the weight term is easy, since the weak convergence $w_h  \rightharpoonup w$ implies the uniform convergence of the midlines. 

\begin{proposition}\label{cont-peso}
For any $w \in W$ and for any sequence $(w_h)$ in $W$ with $w_h \rightharpoonup w$ in $L^p([0,L];\R^3)$ we have 
\begin{equation}\label{lim-w}
\lim_{h\to+\infty}E_{\e_h}^{\rm g}(w_h)=\int_0^L|\mathcal A(s)|\rho_0(s)\vett g\cdot \vett x[w](s)\,ds.
\end{equation}
\end{proposition}

\begin{proof}
By the change of variables $\zeta_i=\e_h\eta_i$, $i=1,2$, we obtain 
\[
\begin{aligned}
\frac{1}{\e_h^2}&\int_{\Omega_{\e_h}}\rho(s,\zeta_1,\zeta_2)\vett g\cdot \vett p_{\e_h}[w_h](s,\zeta_1,\zeta_2)\,dsd\zeta_1d\zeta_2\\
&=\frac{1}{\e_h^2}\int_{\Omega_{\e_h}}\rho(s,\zeta_1,\zeta_2)\vett g\cdot (\vett x[w_h](s)+\zeta_1\vett d[w_h](s)+\zeta_2\vett t[w_h](s)\times \vett d[w_h](s))\,dsd\zeta_1d\zeta_2\\
&=\int_{\Omega_1}\rho(s,\e_h\eta_1,\e_h\eta_2)\vett g\cdot (\vett x[w_h](s)+\e_h\eta_1\vett d[w_h](s)+\e_h\eta_2\vett t[w_h](s)\times \vett d[w_h](s))\,dsd\eta_1d\eta_2.
\end{aligned}
\]
Passing to the limit as $h\to+\infty$, using the fact that $\vett x[w_h]\to \vett x[w]$ uniformly on $[0,L]$ and applying the Dominated Convergence Theorem we conclude.
\end{proof}

Now we pass to the limit in the soap film part of the energy.  First of all we need the following Theorem whose proof requires just minor modifications of the proof of Theorem \ref{primo2}.

\begin{theorem}
\label{primo3}
Let $(w_h)$ be a sequence in $W$ with $w_h \rightharpoonup w$ in $W$ for some $w\in W$. Assume that
\begin{itemize}
\item[\rm(a)] $\forall h\in\N, S_h\in \mathcal{P}(\Lambda_{\e_h}[w_h],\mathcal C_{\e_h,w_h})$;
\item[\rm(b)] $\sup_{h \in \N}\FF(S_h) = \sup_{h \in \N} \inf \{\FF(S) : S\in \mathcal{P}(\Lambda_{\e_h}[w_h],\mathcal C_{\e_h,w_h})\}<+\infty.$
\end{itemize}
Let $\mu_h:= F\mathcal H^2\res S_h$. Then the following three statements hold true:
\begin{equation}\label{c1}
\mu_h \rightharpoonup^{*} \mu \quad \textrm{(up to subsequences)},
\end{equation}
\begin{equation}\label{c2}
\mu \geq F \mathcal H^2\res S_{\infty},\hbox{ where $S_{\infty}= (\hbox{supt}\, \mu) \setminus \vett x[w]([0,L])$ is $2$-rectifiable},
\end{equation}
\begin{equation}\label{c3}
S_\infty \in \mathcal P(\vett x[w]([0,L]),\mathcal C_w).
\end{equation}
\end{theorem} 

Now we prove the existence of a recovery sequence.

\begin{proposition}\label{strong}
Consider $w \in W$ and  $(w_{k})\subset W$ such that $w_k\rightharpoonup w$ in $W$. For any $\e_h\to 0$, there exists $(w_{k_h})$ subsequence of $(w_k)$ such that 
\begin{equation}\label{limsup-sf}
\inf\{\FF(S) : \textrm{$S\in\mathcal P(\vett x[w]([0, L]),\mathcal C_w)$}\} \ge \limsup_{h\to+\infty}E^{\rm sf}_{\e_h}(w_{k_h}). 
\end{equation}
\end{proposition}

\begin{proof}
Since $w_k\rightharpoonup w$ in $W$, $\vett x[w_h] \to \vett x[w]$ uniformly on $[0,T]$. Then for every $h \in \N$ there exists $k_h \in \N$ such that 
\begin{equation}\label{class}
\|\vett x[w_{k_h}]-\vett x[w]\|_\infty\le \frac{\e_h}{2}.
\end{equation}
Since we can assume without loss of generality that
\[
\inf\{\FF(S) : \textrm{$S\in\mathcal P(\vett x[w]([0, L]),\mathcal C_w)$}\} <+\infty,
\]
again applying Theorem 2.7 of \cite{DDG}, we find $S_\infty\in \mathcal P(\vett x[w]([0,L]),\mathcal C_w)$ such that 
\[
\FF(S_\infty)=\min \,\{ \FF(S): S \in \mathcal P(\vett x[w]([0,L]),\mathcal C_w)\}.
\]
Now we set 
\[
S_h:=S_\infty\setminus \Lambda_{\e_h}[w_{k_h}]. 
\]
For any $\gamma \in C(\Lambda_{\e_h}[w_{k_h}])$ not homotopic to a point in $\R^3\setminus \Lambda_{\e_h}[w_{k_h}]$ we have 
\[
(S_\infty\setminus \Lambda_{\e_h}[w_{k_h}])\cap \gamma(\mathbb S^1)\ne \emptyset. 
\]
As a consequence, 
\[
\begin{aligned}
\limsup_{h\to+\infty}E_{\e_h}^{\rm sf}(w) &\le \limsup_{h\to+\infty}\mathcal\FF(S_h)\le  \mathcal \FF(S_\infty)=\min\{\mathcal \FF(S) : \textrm{$S\in \P(\vett x[w]([0,L]),\mathcal C_w)$}\},
\end{aligned}
\]
which concludes the proof.
\end{proof}

\begin{proof}
The compactness statement is Proposition \ref{prop-comp}. Inequality \eqref{lower} follows combining \eqref{lim-w} and \eqref{c2} with the subadditivity of the liminf operator. Next, for any $w\in W$, we consider the constant sequence $w_h\equiv w$. Applying Proposition \ref{strong}, for every $\e_h\to 0$, the (unique) subsequence $\bar w_h \equiv w$ of $(w_h)$ satisfies obviously $\bar w_h \rightharpoonup w$ in $L^p([0,L];\R^3)$ and \eqref{limsup-sf}. Inequality \eqref{upper} follows easily combining \eqref{lim-w} and \eqref{limsup-sf} with the superadditivity of the limsup operator.
\end{proof}

%%%%%%%  SECTION 4  %%%%%%%%%%%%%%%%%%%%%%%%%%%%%%%

\end{document}